\pdfoutput=1
\documentclass[10pt, a4paper]{article}
\usepackage[english]{babel}
\usepackage{amsmath}
\usepackage{amsthm,amssymb,mathrsfs,setspace,amsmath}
\usepackage{epsfig}
\usepackage{mleftright}
\usepackage{float}
\usepackage{subfig}
\usepackage{color}
\usepackage[a4paper, top=1in, bottom=1in, left=1in, right=1in]{geometry}
\usepackage[numbers]{natbib}
\usepackage{lineno}
\numberwithin{figure}{section}
\newtheorem{theorem}{Theorem}[subsection]
\newtheorem{lemma}{Lemma}[subsection]

\newtheorem{example}{Example}[subsection]

\newtheorem{definition}{Definition}[subsection]

\begin{document}

\title{Edge chromatic index and edge-sum chromatic  index  for  families of integral sum graphs}
\author{Priyanka B R$^a$\thanks{$^{a}$priyankabr999@gmail.com},  Biswajit Pandit$^b$ \thanks{$^b$ biswajitpandit82@gmail.com } Rajeshwari M.$^c$ \thanks{$^b$ rajeshwarim@presidencyuniversity.in } 
\\\small{\textit{$^{a,c}$Department of Mathematics, Presidency University Bangalore, $560064$, India.}}
}
\date{\today}

\maketitle

\begin{abstract}
 We consider class of integral sum graphs    $H^{-i,s}_{m,j}$ subject to the conditions     $-i<0<s$, $1\leq m < i$ and $1\leq j < s$ for all $i,s, m,j\in \mathbb{N}$. We apply  edge-sum coloring and edge coloring on    $H^{-i,s}_{m,j}$. Since the graphs fully depend on $i$ and $s$, therefore it is not easy to derive the theoretical as well as numerical results  for all values of $i$ and $s$. Here, we derive the general formula for computing the minimum number of independent color classes.   We compute the edge chromatic as well as  edge-sum chromatic number of   $H^{-i,s}_{m,j}$ corresponding to  different values of $i$, $s$, $m$ and $j$.  We also compare these two techniques. We place the numerical results to verify the theoretical results. 
\end{abstract}
\textit{Keywords:} Edge Coloring; Edge Sum Coloring; Integral Sum Graph; Graph Coloring. \\
\textit{AMS Subject Classification:} 05C15; 05C35.

\section{Introduction}
A graph $G$ comprises a finite nonempty set $V$ of objects which are called vertices and a set $E$ which are called edges.  The order refers to the total number of vertices and the size refers to the total number of edges of a graph. The total number of adjacent edges in a graph corresponding to the vertex is said to be the degree of that vertex  $v\in V(G)$. The highest degree among the degree of vertices of $G$ is denoted by $\Delta(G)$.  

There are two types of graphs:  directed graphs and undirected graphs. In undirected graphs, the edges do not have any direction. Here, we will focus on undirected families of integral sum graphs.  The sum graph (\citep{Hararay1990, Hararay1969, Hararay1994, TNSSVV2001}) $G^{+}(S)$ of a finite subset $S\subset \mathbb{N}=\lbrace 1,2,\cdots \rbrace$ is the graph $(V, E)$, where $V=S$ and $uv\in E$ if and only if $u+v \in S.$ If we consider $S=\lbrace 1,2,\cdots, n \rbrace$, then the sum graph $G^{+}(S)$ is denoted by $G_{n}$.  An integral sum graph (\citep{chen2006, Xu1994, WU2003, MP2002})  $G^{+}(S)$ is defined  as the sum graph along with the range  $S\subset \mathbb{Z}$.  If we consider  $S=\left[ -i,s\right]=\lbrace -i, -(i+1),\cdots,0,\cdots, s-1, s \rbrace$ for all $i,~s\in \mathbb{N}$, then we denote the  integral sum graph as $G_{-i,s}$.   Many authors (\citep{VVLMF2013, LWBSMH2021})  introduced the various properties of different types of integral sum graphs. 
 
Now, we denote (\citep{VVLMF2013}) the edge-sum color class of integral sum graphs as $\left[e_{k}\right]$ corresponding to  edge-sum label $k ~\in S $. We define the vertex $u_{k}$ whose label is $k$ in $G_{-i,s}$. In the following,  we  define  the  graphs
\begin{eqnarray}
\label{paper2eq3}
H_{m,j}^{-i,s}=G_{-i,s} \backslash \left(\lbrace u_{-m}, u_{j}\rbrace \cup \left[e_{-m}\cup e_{j}\right] \right),
\end{eqnarray}
 such that  $1\leq m < i$ and $1 \leq j < s$ for all $i,s, m,j\in \mathbb{N}$. The graphs in   (\ref{paper2eq3}) are obtained from the integral sum graphs $G_{-i,s}$ by removing one vertex and corresponding to its  edge sum color class.    Therefore, we can see 	   $H_{m,j}^{-i,s}$ are the integral sum graphs having vertices   $[-i,s]\backslash \lbrace -m, j  \rbrace$  for all $1\leq m < i $ and $1 \leq j < s$.

Based on above motivations, here we  apply  edge-sum coloring and edge coloring on a class of  integral sum graphs   $H_{m,j}^{-i,s}$ where $i,s, m,j\in \mathbb{N}$ such that $-i<0<s$, $1\leq m<i$ and $1\leq j <s$. Since the graphs fully depend on $i$ and $s$, therefore it is not easy to determine the minimum number of color classes for all values of $i$ and $s$. Here, we derive the general formula for computing the minimum number of independent color classes. We compute the edge chromatic as well as an edge-sum chromatic number to compare these two techniques. We place the numerical results to verify the theoretical result.  To the best of our knowledge, there are no works in the literature.  Below, we discuss the possible related works in the literature corresponding to our methodology.

In $1970$, Folkman (\citep{JF1970}) derived some properties of monochromatic complete graphs with the help of edge coloring.  After that, Baranyai (\citep{ZB1979}) studied the edge-coloring and factorization problems of different types of hypergraphs such as regular, complete,  $r$-partite, and $h$-uniform hypergraphs having $m$ vertices. Holyer (\citep{IH1981}) computed the chromatic index of the NP-complete problem by applying edge coloring.   In the subsequent year,    Cole et al. (\citep{CH1982}) developed an algorithm based on edge coloring to compute the different properties of a bipartite graph. Later Goldberg (\citep{MKG1984}) considered the multigraphs to apply the edge coloring technique.    Consecutively, a generalization of edge-coloring based on the multi-graphs was developed by Hakimi (\citep{SLH1986}).   Panconesi et al. (\citep{APAS1997}) developed fast and simple randomized algorithms for edge coloring problems that arise in a real life.  Kapoor (\citep{AKRR1999}) studied the edge coloring of a bipartite graph having  $m$ edges and $n$ nodes. Some network problems could also be converted to the strong edge coloring and adjacent strong edge coloring problem,  which was developed by Zhang et al. (\citep{ZZLLJW2002}).

We organize the remaining work in the following sections. In section $2$, we have stated  some basic preliminaries  which help us to generalize the results.  In section \ref{section5}, we have computed the edge chromatic index of $H^{-i,s}_{m,\phi}$. We have placed the numerical results in section $4$. In section $5$, we conclude the work.

\section{Preliminaries}
Harary (\citep{Hararay1969, Hararay1990}) introduced the families of sum graphs $G_{n} $ and families of integral sum graphs $G_{-i,s}$ for all $i,~s\in \mathbb{Z}$.  Recenlty, Vilfred et al. (\citep{VVLMF2013}), derived some properties of   $H_{m,j}^{-i,s}$.  Now, we state some useful lemmas without proof which help us to apply edge coloring and edge-sum coloring.
\begin{theorem}(\citep{VLMF2012})
Let $u_{i}$ denotes the vertex whose label is $i$ in the sum graph $G_{n}$, then 
\begin{equation}
\mbox{deg($i$)} = 
\left\{
    \begin{array}{lr}
        n-i-1, & \text{if } 1\leq i \leq \left\lfloor \frac{n}{2}\right\rfloor,\\
        n-i, & \text{if } \left\lfloor \frac{n}{2}\right\rfloor < i\leq n.
    \end{array}
\right.
\end{equation}
\end{theorem}
\begin{theorem}(\cite{VLMF2012})\label{theorem12451}
Let $u_{i}$ denotes the vertex whose label is $i$ in the integral  sum graph $G_{r,s}$ such that $r<0<s$, then 
\begin{equation}
\mbox{deg($i$)} = 
\left\{
    \begin{array}{lr}
        n-i-1, & \text{if } 1\leq i \leq \left\lfloor \frac{s}{2}\right\rfloor,\\
        n-i, & \text{if } \left\lfloor \frac{s}{2}\right\rfloor < i\leq s,\\
        n-1, & \text{if } i=0,\\
        n+i-1, & \text{if } 1\leq i \leq \left\lfloor \frac{-r}{2}\right\rfloor,\\
         n+i, & \text{if } \left\lfloor \frac{-r}{2}\right\rfloor < i\leq -r.
    \end{array}
\right.
\end{equation}
\end{theorem}
\begin{theorem}(\citep{VLMF2012})\label{paper2theorem103}
Let $|r|+s \geq 3$, then the number of edges in the integral sum graph $G_{r,s}$ is 
\begin{equation}
E\left(G_{r,s}\right)=\frac{1}{4}\left(r^{2}+s^{2}-3r+3s-4rs\right)-\frac{1}{2}\left(\left\lfloor\frac{|r|}{2}\right\rfloor+\left\lfloor\frac{s}{2}\right\rfloor \right).
\end{equation}
\end{theorem}

\begin{definition}
Edge-sum chromatic  number is defined as the minimal number of  non-empty independent edge-sum  color classes of a integral sum graphs $G(s)$,  and   is represented by $\chi'_{\mathbb{Z}-sum}\left(G(s)  \right)$.
\end{definition}
\begin{definition}
Edge chromatic  number is defined as the minimal number of  non-empty independent edge color classes of integral sum graph $G(s)$, and   is represented by $\chi'\left( G(s) \right)$.
\end{definition}
\begin{definition}\label{paper2def3}
Let $G(s)$  be the integral sum graph. It  is called perfect edge sum color graph if $\chi'_{\mathbb{Z}-sum}\left( G(s)\right)$ is equal to $\chi'\left(G(s) \right)$.    Otherwise, $G(s)$ is said to be  non-perfect edge sum color graph.
\end{definition}
\begin{example}
$\chi'_{Z-sum}\left( H_{1,2}^{-4,3} \right)=6$ and $\chi'\left( H_{1,2}^{-4,3} \right)=5.$ The integral sum graphs $H_{1,2}^{-4,3}$ is called the non-perfect edge sum color graphs.
\end{example}
In the following, we derive results regarding to chromatic index of different types of integral sum graphs  $H^{-i,s}_{m,j}$.

\section{Chromatic index of $H^{-i,s}_{m,j}$}\label{section5}
Here, we compute the chormatic index for the graph $H^{-i,s}_{m,j}$.  
\begin{lemma}\label{2paper1lemma5000}
The edge chromatic number of  $H_{1,1}^{-2,s}$ for $s=3,\cdots,6$, $H_{1,2}^{-2,s}$ for $s=3,\cdots,8$ and  $H_{1,3}^{-2,s}$ for $s=3,\cdots,10$ is $s.$  
\end{lemma}
\begin{proof}
The minimum number of edge color classes of $H_{1,1}^{-2,3}$, $H_{1,2}^{-2,3}$ and $H_{1,3}^{-2,3}$ are 
$\lbrace \lbrace \left[-2,0\right] \rbrace$, $\lbrace \left[-2,2\right], \left[0,3\right] \rbrace$,  $\lbrace \left[0,3\right] \rbrace \rbrace$, $\lbrace \lbrace \left[-2,0\right] \rbrace$, $\lbrace \left[-2,3\right], \left[0,1\right] \rbrace$,  $\lbrace \left[0,3\right] \rbrace \rbrace$ and $\lbrace \lbrace \left[-2,0\right] \rbrace$, $\lbrace \left[-2,2\right], \left[0,1\right] \rbrace$,  $\lbrace \left[0,2\right] \rbrace \rbrace$, respectively. Hence, edge chormatic number is $3$. Similarly, we can show that the edge chromatic index of $H_{1,1}^{-2,s}$ for $s=4,\cdots,6$, $H_{1,2}^{-2,s}$ for $s=4,\cdots,8$ and  $H_{1,3}^{-2,s}$ for $s=4,\cdots,10$ is $s.$  
\end{proof}
In view of Lemma \ref{2paper1lemma5000}, we have proved the following three lemmas.
\begin{lemma}\label{00paper1lemma500}
The edge chromatic number of  $H_{m,1}^{-3,s}$   for $s=4,\cdots,7$, $H_{m,2}^{-3,s}$  for $s=4,\cdots,9$ and  $H_{m,3}^{-3,s}$   for $s=4,\cdots,11$ is $s+1$ for $m=1,2$.  
\end{lemma}
\begin{lemma}\label{paper1lemma500000}
The edge chromatic number of  $H_{1,1}^{-i,2}$ for $i=3,\cdots,6$, $H_{2,1}^{-i,2}$ for $i=3,\cdots,8$ and  $H_{3,1}^{-i,2}$ for $i=3,\cdots,10$ is $i.$  
\end{lemma}
\begin{lemma}\label{paper1lemma50000}
The edge chromatic number of  $H_{1,j}^{-i,3}$  for $i=4,\cdots,7$, $H_{2,j}^{-i,3}$ for $i=4,\cdots,9$ and  $H_{3,j}^{-i,3}$   for $i=4,\cdots,11$ is $i+1$ for $j=1,2$.  
\end{lemma}
\begin{theorem}\label{0paper2theorem10}
Let $H_{m,j}^{-2,s}$ be the integral sum graphs for all $s\geq3 $, $m=1$ and $j=1,2$. Then these are non-perfect edge sum color graphs for all $s\geq3.$ 
\end{theorem}
\begin{proof}
By similar analysis, we have independent edge sum color classes of $H_{1,j}^{-2,s}$ are  $E_{-2}$, $E_{0}$, $E_{1}, \cdots $, $ E_{j-1}$, $E_{j+1}, \cdots $, $E_{s}$. So, the edge sum chromatic index is $s+1.$ Now, we define the following edge color classes:\\
{\em Case A: $m=1$, $j=1$ and $s\geq7$:}
\begin{eqnarray}\label{2paper1eq1110124}
&&\nonumber \left\lbrace (0,-2), ~(2,k-3),~ (3,k-4), ~ \cdots, ~ \left(\left\lfloor\frac{k-2}{2}\right\rfloor, \left\lfloor\frac{k+1}{2}\right\rfloor\right) \right\rbrace,\\
&&\nonumber \left\lbrace (0,k-1), ~(2,k-2),~ (3,k-3)), ~ \cdots, ~ \left(\left\lfloor\frac{k-1}{2}\right\rfloor, \left\lfloor\frac{k+2}{2}\right\rfloor\right) \right\rbrace,\\
&& \left\lbrace(0,k),~ (-2,2) \right\rbrace;\mbox{ where } k=s,\\
&&\nonumber \left\lbrace (-2,k), ~(0,k-2) \right\rbrace; \mbox{ where } k=4,5,6,\\
&&\nonumber \left\lbrace (-2,k), ~(0,k-2), ~(2,k-4),~ (3,k-5), ~\cdots,~ \left(\left\lfloor\frac{k-3}{2}\right\rfloor, \left\lfloor\frac{k}{2}\right\rfloor\right) \right\rbrace;\\
&&\nonumber \hspace{2cm} \mbox{where } k=7,8,9, ~\cdots~ , s.
\end{eqnarray}
{\em Case B: $m=1$, $j=2$ and $s\geq9$:}
\begin{eqnarray}\label{2paper1eq121224}
&&\nonumber \left\lbrace (0,-2), ~(1,k-2),~ (3,k-4),~(4,k-5), ~ \cdots, ~ \left(\left\lfloor\frac{k-2}{2}\right\rfloor, \left\lfloor\frac{k+1}{2}\right\rfloor\right) \right\rbrace,\\
&&\nonumber \left\lbrace(0,k), ~(1,k-1), ~(3,k-3), ~(4,k-4), ~\cdots, ~\left(\left\lfloor\frac{k-1}{2}\right\rfloor, \left\lfloor\frac{k+2}{2}\right\rfloor\right) \right\rbrace,\\
&& \left\lbrace (0,k-1), \right\rbrace; \mbox{ where } k=s, \\ 
&&\nonumber \left\lbrace (-2,3), ~(0,1) \right\rbrace, \left\lbrace(-2,5),   ~(0,3)  \right\rbrace,\\
&&\nonumber \left\lbrace(-2,k),  ~(0,k-2), ~(1,k-3) \right\rbrace;\mbox{ where } k=6,7,8, \\
&&\nonumber\left\lbrace (-2,k),   ~(0,k-2), ~(1,k-3), ~(3,k-5), ~(4,k-6), ~(5,k-7), ~\cdots\right.\\
&& \nonumber \hspace{3.5cm} \left.\cdots, ~\left(\left\lfloor\frac{k-3}{2}\right\rfloor, \left\lfloor\frac{k}{2}\right\rfloor\right)\right\rbrace; \mbox{ where } k=9,10,11,12, ~\cdots~ , s.
\end{eqnarray}
{\em Case C: $m=1$, $j=3$ and $s\geq11$:}
\begin{eqnarray}\label{2paper1eq1212214}
&&\nonumber \left\lbrace (0,-2), ~(1,k-2),~ (2,k-3),~(4,k-5), ~ \cdots, ~ \left(\left\lfloor\frac{k-2}{2}\right\rfloor, \left\lfloor\frac{k+1}{2}\right\rfloor\right) \right\rbrace,\\
&&\nonumber \left\lbrace(0,k), ~(1,k-1), ~(2,k-2), ~(4,k-4), ~\cdots, ~\left(\left\lfloor\frac{k-1}{2}\right\rfloor, \left\lfloor\frac{k+2}{2}\right\rfloor\right) \right\rbrace,\\
&& \left\lbrace (0,k-1), ~(-2,2) \right\rbrace; \mbox{ where } k=s, \\
&& \left\lbrace (0,1) \right\rbrace, \left\lbrace (-2,4), ~(0,2) \right\rbrace, \nonumber \left\lbrace (-2,6), ~(0,4) \right\rbrace,  \left\lbrace (-2,7), ~(0,5), ~(1,4) \right\rbrace, \\
&&\nonumber \left\lbrace(-2,k),  ~(0,k-2), ~(1,k-3),~(2,k-4) \right\rbrace;\mbox{ where } k=8,9,10. \\
&&\nonumber\left\lbrace (-2,k),  ~(0,k-2), ~(1,k-3), ~(2,k-4), ~(4,k-6), ~(5,k-7), ~\cdots\right.\\
&& \nonumber \hspace{3.5cm} \left.\cdots, ~\left(\left\lfloor\frac{k-3}{2}\right\rfloor, \left\lfloor\frac{k}{2}\right\rfloor\right)\right\rbrace; \mbox{where } k=11,12, ~\cdots~ , s.
\end{eqnarray}
So, from Lemma \ref{2paper1lemma5000} and equations (\ref{2paper1eq1110124}), (\ref{2paper1eq121224}), (\ref{2paper1eq1212214}), we have  
\begin{center}
$\chi'\left(H_{m,j}^{-2,s}\right)=s$ for all $s\geq 3 $, $m=1$ and $j=1,2,3,$
\end{center}
which satisfies the Vizing's theorem.    Hence, the integral sum graph $H_{m,j}^{-2,s}$  are  non-perfect edge-sum color graphs for all $s  \geq   3$, $m=1$ and $j=1,2,3$.
\end{proof}
\begin{theorem}\label{paper2theorem13}
Let $H_{m,j}^{-3,s}$ be the integral sum graphs for all $s\geq4 $, $m=1,2$ and $j=1,2,3$. Then these are non-perfect edge sum color graphs for all $s\geq4.$ 
\end{theorem}
\begin{proof}
By similar analysis, we have the edge sum chromatic index is $s+2.$ Now, we apply edge coloring on these graphs.\\
{\em Case A: $m=1$, $j=1$  and $s\geq8$:}  
\begin{eqnarray}
&&\nonumber \left\lbrace  (0,-3), ~(2,k-4),~ (3,k-5), ~ \cdots, ~ \left(\left\lfloor\frac{k-3}{2}\right\rfloor, \left\lfloor\frac{k}{2}\right\rfloor\right) \right\rbrace,
\end{eqnarray}
\begin{eqnarray}\label{paper1eq1115421}
&&\nonumber \left\lbrace (0,-2), ~(2,k-3),~ (3,k-4), ~ \cdots, ~ \left(\left\lfloor\frac{k-2}{2}\right\rfloor, \left\lfloor\frac{k+1}{2}\right\rfloor\right) \right\rbrace,\\
&&\nonumber \left\lbrace (0,k-1), ~(2,k-2),~ (3,k-3), ~ \cdots, ~ \left(\left\lfloor\frac{k-1}{2}\right\rfloor, \left\lfloor\frac{k+2}{2}\right\rfloor\right) \right\rbrace,\\
&& \left\lbrace(0,k),~ (-2,2),~(-3,3) \right\rbrace,  \left\lbrace ~(-2,k),~ (0,k-2) \right\rbrace;\mbox{ where } k=s,\\
&&\nonumber \left\lbrace (-3,k), ~(-2,k-1), ~(0,k-3) \right\rbrace; \mbox{where } k=5,6,7.\\
&&\nonumber \left\lbrace (-3,k), ~(-2,k-1),~(0,k-3), ~(2,k-5),~ (3,k-6), ~\cdots,~ \left(\left\lfloor\frac{k-4}{2}\right\rfloor, \left\lfloor\frac{k-1}{2}\right\rfloor\right) \right\rbrace;\\ &&\nonumber \hspace{3.5cm}\mbox{where } k=8,9, ~\cdots~ , s.
\end{eqnarray}
{\em Case B: $m=2$, $j=1$ and $s\geq8$:}
\begin{eqnarray}\label{paper1eq111145452}
&&\nonumber \left\lbrace  (0,-3), ~(2,k-4),~ (3,k-5), ~ \cdots, ~ \left(\left\lfloor\frac{k-3}{2}\right\rfloor, \left\lfloor\frac{k}{2}\right\rfloor\right) \right\rbrace,\\
&&\nonumber \left\lbrace (-1,k-1),~(0,k-2), ~(2,k-3),~ (3,k-4), ~ \cdots, ~ \left(\left\lfloor\frac{k-2}{2}\right\rfloor, \left\lfloor\frac{k+1}{2}\right\rfloor\right) \right\rbrace,\\
&&\nonumber \left\lbrace (0,-1), ~(2,k-2),~ (3,k-3), ~ \cdots, ~ \left(\left\lfloor\frac{k-1}{2}\right\rfloor, \left\lfloor\frac{k+2}{2}\right\rfloor\right) \right\rbrace,\\
&& \left\lbrace(0,k), ~(-3,3) \right\rbrace,  \left\lbrace (-3,2),~(-1,k),~ (0,k-1) \right\rbrace,\mbox{ where } k=s,\\
&&\nonumber \left\lbrace (-3,k),  ~(-1,k-2),~(0,k-3) \right\rbrace; \mbox{ where } k=5,6,7,\\
&&\nonumber \left\lbrace (-3,k), ~(-1,k-2),~(0,k-3), ~(2,k-5),~ (3,k-6), ~\cdots,~ \left(\left\lfloor\frac{k-4}{2}\right\rfloor, \left\lfloor\frac{k-1}{2}\right\rfloor\right) \right\rbrace;\\ &&\nonumber \hspace{3.5cm}\mbox{where } k=8,9, ~\cdots~ , s.
\end{eqnarray}
{\em Case C: $m=1$, $j=2$ and $s\geq10$:}
\begin{eqnarray}\label{paper1eq11112}
&&\nonumber \left\lbrace (0,-3), ~(1,k-3),~ (3,k-5), ~ \cdots, ~ \left(\left\lfloor\frac{k-3}{2}\right\rfloor, \left\lfloor\frac{k}{2}\right\rfloor\right) \right\rbrace,\\
&&\nonumber \left\lbrace (0,-2), ~(1,k-2),~ (3,k-4), ~ \cdots, ~ \left(\left\lfloor\frac{k-2}{2}\right\rfloor, \left\lfloor\frac{k+1}{2}\right\rfloor\right) \right\rbrace,\\
&&\nonumber \left\lbrace (0,k), ~(1,k-1),~ (3,k-3), ~ \cdots, ~ \left(\left\lfloor\frac{k-1}{2}\right\rfloor, \left\lfloor\frac{k+2}{2}\right\rfloor\right) \right\rbrace,\\
&& \left\lbrace(0,k-1),~(-3,3) \right\rbrace,  \left\lbrace ~(-3,1),~(-2,k),~ (0,k-2) \right\rbrace;\mbox{ where } k=s,\\
&&\nonumber \left\lbrace (-3,4),~(-2,3),~ (0,1) \right\rbrace,  \left\lbrace (-3,6),~(-2,5), ~(0,3) \right\rbrace,\\
&&\nonumber \left\lbrace (-3,k), ~(-2,k-1), ~(0,k-3),~(1,k-4) \right\rbrace; \mbox{ where } k=7,8,9,\\
&&\nonumber \left\lbrace (-3,k), ~(-2,k-1), ~(0,k-3), ~(1,k-4),~ (3,k-6), ~\cdots,~ \left(\left\lfloor\frac{k-4}{2}\right\rfloor, \left\lfloor\frac{k-1}{2}\right\rfloor\right) \right\rbrace;\\ &&\nonumber \hspace{3.5cm}\mbox{where } k=10,11, ~\cdots~ , s.
\end{eqnarray}
 {\em Case D: $m=2$, $j=2$ and $s\geq10$:}
\begin{eqnarray}\label{paper1eq111123}
&&\nonumber \left\lbrace   (0,-3), ~(1,k-3),~ (3,k-5), ~ \cdots, ~ \left(\left\lfloor\frac{k-3}{2}\right\rfloor, \left\lfloor\frac{k}{2}\right\rfloor\right) \right\rbrace,\\
&&\nonumber \left\lbrace (-1,k),~(0,k-1), ~(1,k-2),~ (3,k-4), ~ \cdots, ~ \left(\left\lfloor\frac{k-2}{2}\right\rfloor, \left\lfloor\frac{k+1}{2}\right\rfloor\right) \right\rbrace,\\
&&\nonumber \left\lbrace (0,-1), ~(1,k-1),~ (3,k-3), ~ \cdots, ~ \left(\left\lfloor\frac{k-1}{2}\right\rfloor, \left\lfloor\frac{k+2}{2}\right\rfloor\right) \right\rbrace,\\
&& \left\lbrace(0,k),~(-1,1),~(-3,3) \right\rbrace,  \left\lbrace  (-1,k-1),~ (0,k-2) \right\rbrace;\mbox{  where } k=s,\\
&&\nonumber \left\lbrace (-3,4), ~ (0,1) \right\rbrace,  \left\lbrace (-3,6), ~ (-1,4),~(0,3) \right\rbrace,\\
&&\nonumber \left\lbrace (-3,k), ~(-1,k-2),~(0,k-3),~(1,k-4) \right\rbrace; \mbox{ where } k=7,8,9,\\
&&\nonumber \left\lbrace (-3,k), ~(-1,k-2),~(0,k-3), ~(1,k-4),~ (3,k-6), ~\cdots,~ \left(\left\lfloor\frac{k-4}{2}\right\rfloor, \left\lfloor\frac{k-1}{2}\right\rfloor\right) \right\rbrace;\\ &&\nonumber \hspace{3.5cm}\mbox{where } k=10,11, ~\cdots~ , s.
\end{eqnarray}
{\em Case E: $m=1$,  $j=3$ and $s\geq12$:}
\begin{eqnarray}\label{paper1eq1111114585}
&&\nonumber \left\lbrace   (0,-3), ~(1,k-3),~ (2,k-4),~(4,k-6), ~ \cdots, ~ \left(\left\lfloor\frac{k-3}{2}\right\rfloor, \left\lfloor\frac{k}{2}\right\rfloor\right) \right\rbrace,\\
&&\nonumber \left\lbrace (0,-2), ~(1,k-2),~ (2,k-3),~(4,k-5), ~ \cdots, ~ \left(\left\lfloor\frac{k-2}{2}\right\rfloor, \left\lfloor\frac{k+1}{2}\right\rfloor\right) \right\rbrace,\\
&&\nonumber \left\lbrace (0,k), ~(1,k-1),~ (2,k-2),~(4,k-4), ~ \cdots, ~ \left(\left\lfloor\frac{k-1}{2}\right\rfloor, \left\lfloor\frac{k+2}{2}\right\rfloor\right) \right\rbrace,\\
&&   \left\lbrace(0,k-1),~(-2,2) \right\rbrace, \left\lbrace ~(-3,1),~(-2,k),~ (0,k-2) \right\rbrace;\mbox{ where } k=s,\\
&&\nonumber \left\lbrace (-3,4),~ (0,1) \right\rbrace,  \left\lbrace (-3,5),~(-2,4),~ (0,2) \right\rbrace,  \left\lbrace (-3,7),~(-2,6),~(0,4) \right\rbrace,\\
&&\nonumber \left\lbrace (-3,8),~(-2,7), ~(0,5), ~(1,4) \right\rbrace,\\
&&\nonumber \left\lbrace (-3,k), ~(-2,k-1),  ~(0,k-3),~(1,k-4),~(2,k-5) \right\rbrace; \mbox{where } k=9,10,11,\\
&&\nonumber \left\lbrace (-3,k), ~(-2,k-1),  ~(0,k-3), ~(1,k-4),~ (2,k-5),~(4,k-7), ~\cdots \right.\\ 
&&\nonumber \hspace{3.5cm} ~\left. \cdots, \left(\left\lfloor\frac{k-4}{2}\right\rfloor, \left\lfloor\frac{k-1}{2}\right\rfloor\right) \right\rbrace \mbox{ where } k=12,13, ~\cdots~ , s.
\end{eqnarray}
{\em Case F: $m=2$, $j=3$ and $s\geq12$:}
\begin{eqnarray}\label{paper1eq111111784}
&&\nonumber \left\lbrace   (0,-3), ~(1,k-3),~ (2,k-4),~(4,k-6), ~ \cdots, ~ \left(\left\lfloor\frac{k-3}{2}\right\rfloor, \left\lfloor\frac{k}{2}\right\rfloor\right) \right\rbrace,\\
&&\nonumber \left\lbrace (-1,k-1),~(0,k), ~(1,k-2),~ (2,k-3),~(4,k-5), ~ \cdots, ~ \left(\left\lfloor\frac{k-2}{2}\right\rfloor, \left\lfloor\frac{k+1}{2}\right\rfloor\right) \right\rbrace,\\
&&\nonumber \left\lbrace (0,-1), ~(1,k-1),~ (2,k-2)),~(4,k-4), ~ \cdots, ~ \left(\left\lfloor\frac{k-1}{2}\right\rfloor, \left\lfloor\frac{k+2}{2}\right\rfloor\right) \right\rbrace,\\
&&  \left\lbrace(0,k-2),~(-1,1) \right\rbrace,
  \left\lbrace(-3,2), ~(-1,k),~ (0,k-1) \right\rbrace;\mbox{ where } k=s,\\
&&\nonumber \left\lbrace (-3,4),~(-1,2),~ (0,1) \right\rbrace,  \left\lbrace (-3,5), ~ (0,2) \right\rbrace,  \left\lbrace (-3,7), ~ (-1,5),~(0,4) \right\rbrace, \\
&&\nonumber \left\lbrace (-3,8), ~ (-1,6),~(0,5), ~(1,4) \right\rbrace,\\
&&\nonumber \left\lbrace (-3,k),  ~(-1,k-2),~(0,k-3),~(1,k-4),~(2,k-5) \right\rbrace; \mbox{where } k=9,10,11,\\
&&\nonumber \left\lbrace (-3,k),  ~(-1,k-2),~(0,k-3), ~(1,k-4),~ (2,k-5),~(4,k-7), ~\cdots \right.\\ 
&&\nonumber \hspace{3.5cm} ~\left. \cdots, \left(\left\lfloor\frac{k-4}{2}\right\rfloor, \left\lfloor\frac{k-1}{2}\right\rfloor\right) \right\rbrace \mbox{ where } k=12,13, ~\cdots~ , s.
\end{eqnarray}
 By using Lemma \ref{00paper1lemma500},  we have the edge chromatic number is 
\begin{center}
$\chi'\left(H_{m,j}^{-3,s}\right)=s+1$ for all $s\geq 4 $, $m=1,2$ and $j=1,2,3.$
\end{center}
 Hence the proof is complete.
\end{proof}
\begin{theorem}\label{paper2theorem90}
Let $H_{m,j}^{-i,2}$ be the integral sum graphs for all $i\geq3 $, $m=1,2,3$ and $j=1$. Then these are non-perfect edge sum color graph for all $i\geq3.$ 
\end{theorem}
\begin{proof}
Similary, we get the edge sum chromatic index is $i+1.$ Corresponding to different values of $m$ and $j$, we define the following independent edge color classes:\\
{\em Case A: $m=1$, $j=1$ and $i\geq7$:} 
\begin{eqnarray}\label{paper1eq1110124}
&&\nonumber \left\lbrace (0,2), ~(-2,-(k-3)),~ (-3,-(k-4)), ~ \cdots, ~ \left(-\left\lfloor\frac{k-2}{2}\right\rfloor, -\left\lfloor\frac{k+1}{2}\right\rfloor\right) \right\rbrace,\\
&&\nonumber \left\lbrace (0,-(k-1)), ~(-2,-(k-2)),~ (-3,-(k-3)), ~ \cdots, ~ \left(-\left\lfloor\frac{k-1}{2}\right\rfloor, -\left\lfloor\frac{k+2}{2}\right\rfloor\right) \right\rbrace,\\
&& \left\lbrace(0,-k),~ (-2,2) \right\rbrace;\mbox{ where } k=i,\\
&&\nonumber \left\lbrace (2,-k), ~(0,-(k-2)) \right\rbrace; \mbox{ where } k=4,5,6,\\
&&\nonumber \left\lbrace (2,-k), ~(0,-(k-2)), ~(-2,-(k-4)),~ (-3,-(k-5)), ~\cdots,~ \left(-\left\lfloor\frac{k-3}{2}\right\rfloor,- \left\lfloor\frac{k}{2}\right\rfloor\right) \right\rbrace;\\
&&\nonumber \hspace{2cm} \mbox{where } k=7,8,9 ~\cdots~ , i.
\end{eqnarray}
{\em Case B: $m=2$, $j=1$ and $i\geq9$:}
\begin{eqnarray}\label{paper1eq121224}
&&\nonumber \left\lbrace (0,2), ~(-1,-(k-2)),~ (-3,-(k-4)),~(-4,-(k-5)), ~ \cdots, ~ \left(-\left\lfloor\frac{k-2}{2}\right\rfloor, -\left\lfloor\frac{k+1}{2}\right\rfloor\right) \right\rbrace,\\
&&\nonumber \left\lbrace(0,-k), ~(-1,-(k-1)), ~(-3,-(k-3)), ~(-4,-(k-4)), ~\cdots, ~\left(-\left\lfloor\frac{k-1}{2}\right\rfloor, -\left\lfloor\frac{k+2}{2}\right\rfloor\right) \right\rbrace,\\
&&\nonumber \left\lbrace (0,-(k-1)) \right\rbrace; \mbox{ where } k=i, \\ 
&&\nonumber \left\lbrace (2,-3), ~(0,-1) \right\rbrace,   \left\lbrace(2,-5),   ~(0,-3)  \right\rbrace,\\
&&\nonumber \left\lbrace(2,-k),  ~(0,-(k-2)), ~(-1,-(k-3)) \right\rbrace;\mbox{ where } k=6,7,8, \\
&&\nonumber\left\lbrace (2,-k),   ~(-0,-(k-2)), ~(-1,-(k-3)), ~(-3,-(k-5)), ~(-4,-(k-6)), ~(-5,-(k-7)), ~\cdots\right.\\
&& \nonumber \hspace{3.5cm} \left.\cdots, ~\left(-\left\lfloor\frac{k-3}{2}\right\rfloor, -\left\lfloor\frac{k}{2}\right\rfloor\right)\right\rbrace; \mbox{ where } k=9,10, ~\cdots~ , i.
\end{eqnarray}
{\em Case C: $m=3$, $j=1$ and $i\geq11$:}
\begin{eqnarray}
&&\nonumber \left\lbrace (0,2), ~(-1,-(k-2)),~ (-2,-(k-3)),~(-4,-(k-5)), ~ \cdots, ~ \left(-\left\lfloor\frac{k-2}{2}\right\rfloor,- \left\lfloor\frac{k+1}{2}\right\rfloor\right) \right\rbrace,
\end{eqnarray}
\begin{eqnarray}\label{paper1eq1212214}
&&\nonumber \left\lbrace(0,-k), ~(-1,-(k-1)), ~(-2,-(k-2)), ~(-4,-(k-4)), ~\cdots, ~\left(-\left\lfloor\frac{k-1}{2}\right\rfloor, -\left\lfloor\frac{k+2}{2}\right\rfloor\right) \right\rbrace,\\
&&\nonumber \left\lbrace (0,-(k-1)), ~(-2,2) \right\rbrace; \mbox{ where } k=i, \\ 
&& \left\lbrace (0,-1) \right\rbrace,   \left\lbrace (2,-4), ~(0,-2) \right\rbrace,   \left\lbrace (2,-6), ~(0,-4) \right\rbrace,   \left\lbrace (2,-7), ~(0,-5), ~(-1,-4) \right\rbrace, \\
&&\nonumber \left\lbrace(2,-k),  ~(0,-(k-2)), ~(-1,-(k-3)),~(-2,-(k-4)) \right\rbrace;\mbox{ where } k=8,9,10, \\
&&\nonumber\left\lbrace (2,-k),  ~(0,-(k-2)), ~(-1,-(k-3)), ~(-2,-(k-4)), ~(-4,-(k-6)), ~(-5,-(k-7)), ~\cdots\right.\\
&& \nonumber \hspace{3.5cm} \left.\cdots, ~\left(-\left\lfloor\frac{k-3}{2}\right\rfloor, -\left\lfloor\frac{k}{2}\right\rfloor\right)\right\rbrace; \mbox{where } k=11,12, ~\cdots~ , i.
\end{eqnarray}
So, edge chromatic index is 
\begin{center}
$\chi'\left(H_{m,j}^{-i,2}\right)=i$ for all $i\geq 3 $, $m=1,2,3$ and $j=1$
\end{center}
which satisfies the Vizing's theorem.    Hence, the integral sum graph $H_{m,j}^{-i,2}$  are  non-perfect edge-sum color graphs for all $i  \geq   3$, $m=1,2,3$ and $j=1$.
\end{proof}
\begin{theorem}\label{paper2theorem11}
Let $H_{m,j}^{-i,3}$ be the integral sum graphs for all $i\geq4 $, $m=1,2,3$ and $j=1,2$. Then these are non-perfect edge sum color graph for all $i\geq4.$ 
\end{theorem}
\begin{proof}
By similar analysis, we have the edge sum chromatic index is $i+2.$ Now, we apply edge coloring on these graphs. We list below the independent color classes:\\
{\em Case A: $m=1$, $j=1$  and $i\geq8$:}
\begin{eqnarray}\label{paper1eq11154}
&&\nonumber \left\lbrace  (0,3), ~(-2,-(k-4)),~ (-3,-(k-5)), ~ \cdots, ~ \left(-\left\lfloor\frac{k-3}{2}\right\rfloor, -\left\lfloor\frac{k}{2}\right\rfloor\right) \right\rbrace,\\
&&\nonumber \left\lbrace (0,2), ~(-2,-(k-3)),~ (-3,-(k-4)), ~ \cdots, ~ \left(-\left\lfloor\frac{k-2}{2}\right\rfloor, -\left\lfloor\frac{k+1}{2}\right\rfloor\right) \right\rbrace,\\
&&\nonumber \left\lbrace (0,-(k-1)), ~(-2,-(k-2)),~ (-3,-(k-3)), ~ \cdots, ~ \left(-\left\lfloor\frac{k-1}{2}\right\rfloor, -\left\lfloor\frac{k+2}{2}\right\rfloor\right) \right\rbrace,\\
&&  \left\lbrace(0,-k),~ (-2,2),~(-3,3) \right\rbrace,  \left\lbrace ~(2,-k),~ (0,-(k-2)) \right\rbrace;\mbox{ where } k=i,\\
&&\nonumber \left\lbrace (3,-k), ~(2,-k-1), ~(0,-(k-3)) \right\rbrace; \mbox{where } k=5,6,7,\\
&&\nonumber \left\lbrace (3,-k), ~(2,-(k-1)),~(0,-(k-3)), ~(-2,-(k-5)),~ (-3,-(k-6)), ~\cdots,\right.\\
 &&\nonumber \hspace{3.5cm}\left. \left(-\left\lfloor\frac{k-4}{2}\right\rfloor, -\left\lfloor\frac{k-1}{2}\right\rfloor\right) \right\rbrace; ~\mbox{where } k=8,9, ~\cdots~ , i.
\end{eqnarray}
{\em Case B: $m=1$, $j=2$ and $i\geq8$:}
\begin{eqnarray}
&&\nonumber \left\lbrace  (0,3), ~(-2,-(k-4)),~ (-3,-(k-5)), ~ \cdots, ~ \left(-\left\lfloor\frac{k-3}{2}\right\rfloor, -\left\lfloor\frac{k}{2}\right\rfloor\right) \right\rbrace,
\end{eqnarray}
\begin{eqnarray}\label{paper1eq1111454}
&&\nonumber \left\lbrace (1,-(k-1)),~(0,-(k-2)), ~(-2,-(k-3)),~ (-3,-(k-4)), ~ \cdots, ~ \left(-\left\lfloor\frac{k-2}{2}\right\rfloor, -\left\lfloor\frac{k+1}{2}\right\rfloor\right) \right\rbrace,\\
&&\nonumber \left\lbrace (0,1), ~(-2,-(k-2)),~ (-3,-(k-3)), ~ \cdots, ~ \left(-\left\lfloor\frac{k-1}{2}\right\rfloor, -\left\lfloor\frac{k+2}{2}\right\rfloor\right) \right\rbrace,\\
&& \left\lbrace(0,-k), ~(-3,3) \right\rbrace,  \left\lbrace (3,-2),~(1,-k),~ (0,-(k-1)) \right\rbrace,\mbox{ where } k=i,\\
&&\nonumber \left\lbrace (3,-k),  ~(1,-(k-2)),~(0,-(k-3)) \right\rbrace; \mbox{ where } k=5,6,7,\\
&&\nonumber \left\lbrace (3,-k), ~(1,-(k-2)),~(0,-(k-3)), ~(-2,-(k-5)),~ (-3,-(k-6)), ~\cdots,\right.\\ &&\nonumber \hspace{3.5cm}\left. \left(-\left\lfloor\frac{k-4}{2}\right\rfloor, -\left\lfloor\frac{k-1}{2}\right\rfloor\right) \right\rbrace;~\mbox{ where } k=8,9, ~\cdots~ , i.
\end{eqnarray}
{\em Case C: $j=1$, $m=2$ and $i\geq10$:}
\begin{eqnarray}\label{paper1eq111132}
&&\nonumber \left\lbrace (0,3), ~(-1,-(k-3)),~ (-3,-(k-5)), ~ \cdots, ~ \left(-\left\lfloor\frac{k-3}{2}\right\rfloor, -\left\lfloor\frac{k}{2}\right\rfloor\right) \right\rbrace,\\
&&\nonumber \left\lbrace (0,2), ~(-1,-(k-2)),~ (-3,-k-4), ~ \cdots, ~ \left(-\left\lfloor\frac{k-2}{2}\right\rfloor, -\left\lfloor\frac{k+1}{2}\right\rfloor\right) \right\rbrace,\\
&&\nonumber \left\lbrace (0,-k), ~(-1,-(k-1)),~ (-3,-(k-3)), ~ \cdots, ~ \left(-\left\lfloor\frac{k-1}{2}\right\rfloor, -\left\lfloor\frac{k+2}{2}\right\rfloor\right) \right\rbrace,\\
&& \left\lbrace(0,-(k-1)),~(-3,3) \right\rbrace,  \left\lbrace ~(3,-1),~(2,-k),~ (0,-(k-2)) \right\rbrace;\mbox{ where } k=i,\\
&&\nonumber \left\lbrace (3,-4),~(2,-3),~ (0,-1) \right\rbrace,  \left\lbrace (3,-6),~(2,-5), ~(0,-3) \right\rbrace,\\
&&\nonumber \left\lbrace (3,-k), ~(2,-(k-1)), ~(0,-(k-3)),~(-1,-(k-4)) \right\rbrace; \mbox{ where } k=7,8,9,\\
&&\nonumber \left\lbrace (3,-k), ~(2,-(k-1)), ~(0,-(k-3)), ~(-1,-(k-4)),~ (-3,-(k-6)), ~\cdots,\right.\\ &&\nonumber \hspace{3.5cm}\left.\left(-\left\lfloor\frac{k-4}{2}\right\rfloor, -\left\lfloor\frac{k-1}{2}\right\rfloor\right) \right\rbrace;~\mbox{where } k=10,11, ~\cdots~ , i.
\end{eqnarray}
{\em Case D: $m=2$, $j=2$ and $i\geq10$:}
\begin{eqnarray}\label{paper1eq111145}
&&\nonumber \left\lbrace   (0,3), ~(-1,-(k-3)),~ (-3,-(k-5)), ~ \cdots, ~ \left(-\left\lfloor\frac{k-3}{2}\right\rfloor, -\left\lfloor\frac{k}{2}\right\rfloor\right) \right\rbrace,\\
&&\nonumber \left\lbrace (1,-k),~(0,-(k-1)), ~(-1,-(k-2)),~ (-3,-(k-4)), ~ \cdots, ~ \left(-\left\lfloor\frac{k-2}{2}\right\rfloor, -\left\lfloor\frac{k+1}{2}\right\rfloor\right) \right\rbrace,\\
&&\nonumber \left\lbrace (0,1), ~(-1,-(k-1)),~ (-3,-(k-3)), ~ \cdots, ~ \left(-\left\lfloor\frac{k-1}{2}\right\rfloor, -\left\lfloor\frac{k+2}{2}\right\rfloor\right) \right\rbrace,\\
&& \left\lbrace(0,-k),~(-1,1),~(-3,3) \right\rbrace,  \left\lbrace  (1,-(k-1)),~ (0,-(k-2)) \right\rbrace;\mbox{  where } k=i,\\
&&\nonumber \left\lbrace (3,-4), ~ (0,-1) \right\rbrace,  \left\lbrace (3,-6), ~ (1,-4),~(0,-3) \right\rbrace,\\
&&\nonumber \left\lbrace (3,-k), ~(1,-(k-2)),~(0,-(k-3)),~(-1,-(k-4)) \right\rbrace; \mbox{ where } k=7,8,9,\\
&&\nonumber \left\lbrace (3,-k), ~(1,-(k-2)),~(0,-(k-3)), ~(-1,-(k-4)),~ (-3,-(k-6)), ~\cdots,\right.\\ &&\nonumber \hspace{3.5cm} \left. \left(-\left\lfloor\frac{k-4}{2}\right\rfloor, -\left\lfloor\frac{k-1}{2}\right\rfloor\right) \right\rbrace;~\mbox{ where } k=10,11, ~\cdots~ , i.
\end{eqnarray}
{\em Case E: $m=3$,  $j=1$ and $i\geq12$:}
\begin{eqnarray}\label{paper1eq11111145}
&&\nonumber \left\lbrace   (0,3), ~(-1,-(k-3)),~ (-2,-(k-4)),~(-4,-(k-6)), ~ \cdots, ~ \left(-\left\lfloor\frac{k-3}{2}\right\rfloor, -\left\lfloor\frac{k}{2}\right\rfloor\right) \right\rbrace,\\
&&\nonumber \left\lbrace (0,2), ~(-1,-(k-2)),~ (-2,-(k-3)),~(-4,-(k-5)), ~ \cdots, ~ \left(-\left\lfloor\frac{k-2}{2}\right\rfloor, -\left\lfloor\frac{k+1}{2}\right\rfloor\right) \right\rbrace,\\
&&\nonumber \left\lbrace (0,-k), ~(-1,-(k-1)),~ (-2,-(k-2)),~(-4,-(k-4)), ~ \cdots, ~ \left(-\left\lfloor\frac{k-1}{2}\right\rfloor, -\left\lfloor\frac{k+2}{2}\right\rfloor\right) \right\rbrace,\\
&&   \left\lbrace(0,-(k-1)),~(-2,2) \right\rbrace,  \left\lbrace ~(3,-1),~(2,-k),~ (0,-(k-2)) \right\rbrace;\mbox{ where } k=i,\\
&&\nonumber \left\lbrace (3,-4),~ (0,-1) \right\rbrace,  \left\lbrace (3,-5),~(2,-4),~ (0,-2) \right\rbrace,  \left\lbrace (3,-7),~(2,-6),~(0,-4) \right\rbrace,\\
&&\nonumber \left\lbrace (3,-8),~(2,-7), ~(0,-5), ~(-1,-4) \right\rbrace,\\
&&\nonumber \left\lbrace (3,-k), ~(2,-(k-1)),  ~(0,-(k-3)),~(-1,-(k-4)),~(-2,-(k-5)) \right\rbrace; \mbox{where } k=9,10,11,\\
&&\nonumber \left\lbrace (3,-k), ~(2,-(k-1)),  ~(0,-(k-3)), ~(-1,-(k-4)),~ (-2,-(k-5)),~(-4,-(k-7)), ~\cdots \right.\\ 
&&\nonumber \hspace{3.5cm} ~\left. \cdots, \left(-\left\lfloor\frac{k-4}{2}\right\rfloor, -\left\lfloor\frac{k-1}{2}\right\rfloor\right) \right\rbrace; \mbox{ where } k=12,13, ~\cdots~ , i.
\end{eqnarray}
{\em Case F: $m=3$, $j=2$ and $i\geq12$:}
\begin{eqnarray}\label{paper1eq111111}
&&\nonumber \left\lbrace   (0,3), ~(-1,-(k-3)),~ (-2,-(k-4)),~(-4,-(k-6)), ~ \cdots, ~ \left(-\left\lfloor\frac{k-3}{2}\right\rfloor, -\left\lfloor\frac{k}{2}\right\rfloor\right) \right\rbrace,\\
&&\nonumber \left\lbrace (1,-(k-1)),~(0,-k), ~(-1,-(k-2)),~ (-2,-(k-3)),~(-4,-(k-5)), ~ \cdots,\right.\\
&& \nonumber \hspace{2cm}\left. \left(-\left\lfloor\frac{k-2}{2}\right\rfloor, -\left\lfloor\frac{k+1}{2}\right\rfloor\right) \right\rbrace,\\
&&\nonumber \left\lbrace (0,1), ~(-1,-(k-1)),~ (-2,-(k-2)),~(-4,-(k-4)), ~ \cdots, ~ \left(-\left\lfloor\frac{k-1}{2}\right\rfloor,- \left\lfloor\frac{k+2}{2}\right\rfloor\right) \right\rbrace,\\
&&  \left\lbrace(0,-(k-2)),~(-1,1) \right\rbrace, \left\lbrace(3,-2), ~(1,-k),~ (0,-(k-1)) \right\rbrace;\mbox{ where } k=i,\\
&&\nonumber \left\lbrace (3,-4),~(1,-2),~ (0,-1) \right\rbrace,  \left\lbrace (3,-5), ~ (0,-2) \right\rbrace,  \left\lbrace (3,-7), ~ (1,-5),~(0,-4) \right\rbrace,\\
&&\nonumber \left\lbrace (3,-8), ~ (1,-6),~(0,-5), ~(-1,-4) \right\rbrace,\\
&&\nonumber \left\lbrace (3,-k),  ~(1,-(k-2)),~(0,-(k-3)),~(-1,-(k-4)),~(-2,-(k-5)) \right\rbrace; \mbox{ where } k=9,10,11,\\
&&\nonumber \left\lbrace (3,-k),  ~(1,-(k-2)),~(0,-(k-3)), ~(-1,-(k-4)),~ (-2,-(k-5)),~(-4,-(k-7)), ~\cdots \right.\\ 
&&\nonumber \hspace{3.5cm} ~\left. \cdots, \left(-\left\lfloor\frac{k-4}{2}\right\rfloor, -\left\lfloor\frac{k-1}{2}\right\rfloor\right) \right\rbrace \mbox{ where } k=12,13, ~\cdots~ , i.
\end{eqnarray}
 Hence,  we have the edge chromatic number  
\begin{center}
$\chi'\left(H_{m,j}^{-i,3}\right)=i+1$ for all $i\geq 4 $, $m=1,2,3$ and $j=1,2.$
\end{center}
 Finally, we can conclude that the integral sum graph $H_{m,j}^{-i,3}$   is non-perfect edge sum color graphs for every values of  $i  \geq   4$, $m=1,2,3$ and $j=1,2$.
\end{proof}
\section{Numerical results}
Here, we place the numerical results to verify the derived theoretical results. Put $i=11$ in    Theorem \ref{paper2theorem90} to compute the edge chromatic index of $H^{-11,2}_{3,1}$. We assign the color to the edges of $H^{-11,2}_{3,1}$ in the following manner:
\begin{eqnarray*}
&& \lbrace(0,2),~(-1,-9),~(-2,-8),~(-4,-6)\rbrace \rightarrow \mbox{Blue},\\
&& \lbrace(0,-11),~(-1,-10),~(-2,-9),(-4,-7),(-5,-6) \rbrace \rightarrow \mbox{Ivory},  \lbrace(0,-10),~(-2,2)\rbrace \rightarrow \mbox{Brown},\\
&& \lbrace(0,-1) \rbrace \rightarrow \mbox{Orange}, \lbrace(2,-4),(0,-2) \rbrace \rightarrow \mbox{Grey},~
\lbrace(2,-6),(0,-4)\rbrace \rightarrow \mbox{Purple},\\
&& \lbrace(2,-7),(0,-5),(-1,-4)\rbrace \rightarrow \mbox{Green},  \lbrace(2,-8),(0,-6),(-1,-5),(-2,-4)\rbrace \rightarrow \mbox{Yellow},\\
&& \lbrace(2,-9),(0,-7),(-1,-6),(-2,-5)\rbrace \rightarrow \mbox{Pink}, 
\lbrace(2,-10),(0,-8),(-1,-7),(-2,-6)\rbrace \rightarrow \mbox{Red},\\
&&
\lbrace(2,-11),(0,-9),(-1,-8),(-2,-7),(-4,-5)\rbrace \rightarrow \mbox{Black}.
\end{eqnarray*}
Therefore, edge chromatic index  of $H^{-11,2}_{3,1}$ is $11$.  Similarly, for $i=12$, $m=3$ and $j=2$ from Theorem \ref{paper2theorem11}, we compute the edge color classes as follows:
\begin{eqnarray*}
&& \lbrace(0,3),~(-1,-9),~(-2,-8),~(-4,-6)\rbrace \rightarrow \mbox{Purple},\\
&& \lbrace(1,-11),~(0,-12),~(-1,-10),(-2,-9),(-4,-7),(-5,-6) \rbrace \rightarrow \mbox{Musturd},\\
&& \lbrace(0,1),(-1,-11),(-2,-10),(-4,-8),~(-5,-7)\rbrace \rightarrow \mbox{Cyan},  \lbrace(0,-10),(-1,1)  \rbrace \rightarrow \mbox{ivory}, \\
&&\lbrace(3,-2),(1,-12),(0,-11)  \rbrace \rightarrow \mbox{Maroon},\lbrace(3,-4),(1,-2),(0,-1)\rbrace \rightarrow \mbox{Olive green},\\
&& \lbrace(3,-5), (0,-2)\rbrace \rightarrow \mbox{Grey}, \lbrace(3,-7), (1,-5),(0,-4)\rbrace \rightarrow \mbox{Orange},\\
&& \lbrace(3,-8), (1,-6),(0,-5),(-1,-4)\rbrace \rightarrow \mbox{Blue},\\
&&
\lbrace(3,-9), (1,-7),(0,-6),(-1,-5),(-2,-4)\rbrace \rightarrow \mbox{Yellow},\\
&&
\lbrace(3,-10), (1,-8),(0,-7),(-1,-6),(-2,-5)\rbrace \rightarrow \mbox{Pink},\\
&&
\lbrace(3,-11), (1,-9),(0,-8),(-1,-7),(-2,-6)\rbrace \rightarrow \mbox{Red},\\
&&
\lbrace(3,-12), (1,-10),(0,-9),(-1,-8),(-2,-7),(-4,-5)\rbrace \rightarrow \mbox{Black}.
\end{eqnarray*}
Therefore, edge chromatic index of the graph $H^{-12,3}_{3,2}$ is $13$. In   figure \ref{paper1figure4}, we place the graphs of the edge coloring for $H^{-11,2}_{3,1}$ and $H^{-12,3}_{3,2}$.
\begin{figure}[H]
  \subfloat[$H^{-11,2}_{3,1}$]{
	\begin{minipage}[c][1\width]{
	   0.3\textwidth}
	   \centering
	   \includegraphics[width=1\textwidth]{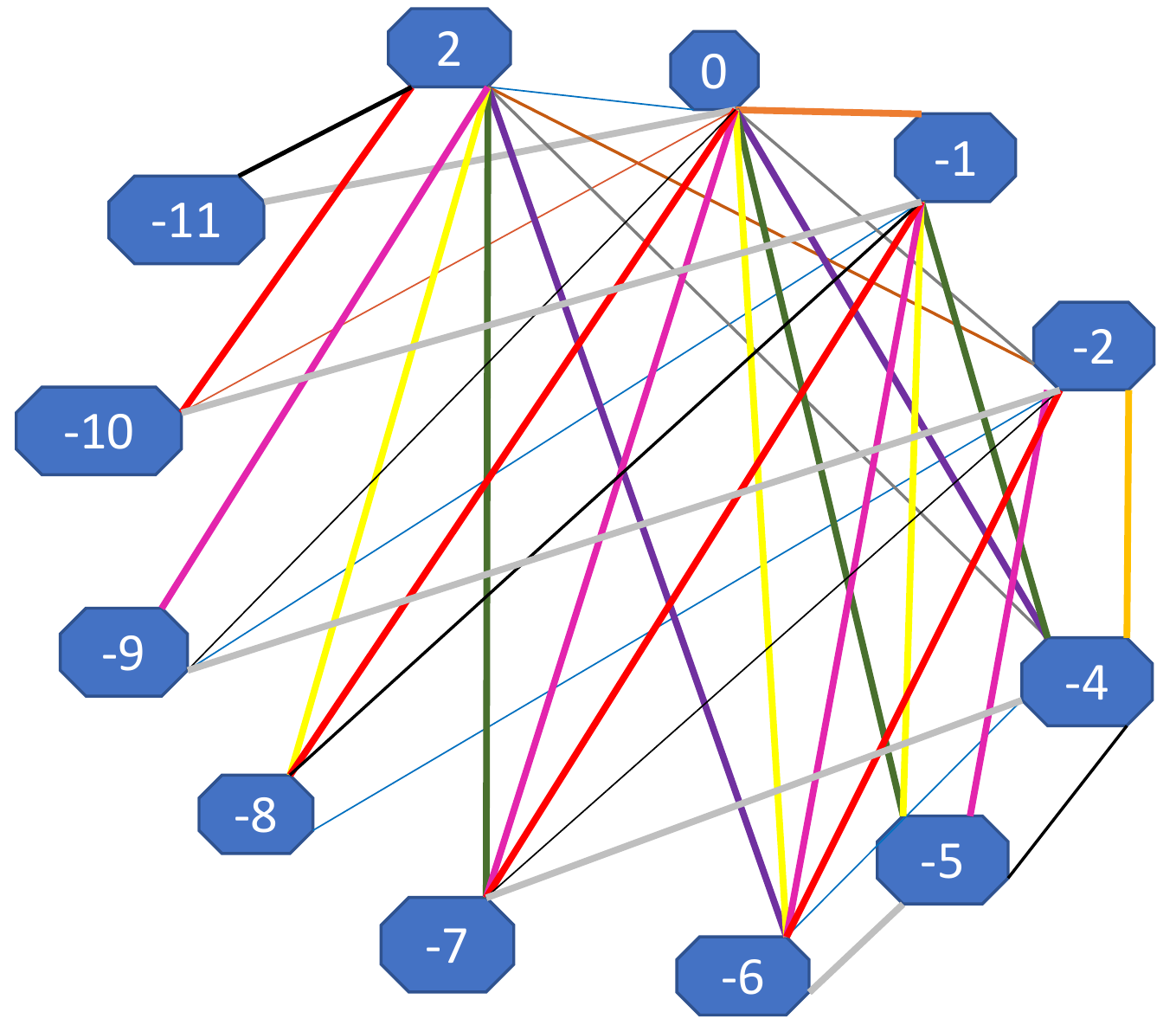}
	\end{minipage}}
 \hspace{3cm}	
  \subfloat[$H^{-12,3}_{3,2}$]{
	\begin{minipage}[c][1\width]{
	   0.3\textwidth}
	   \centering
	   \includegraphics[width=1\textwidth]{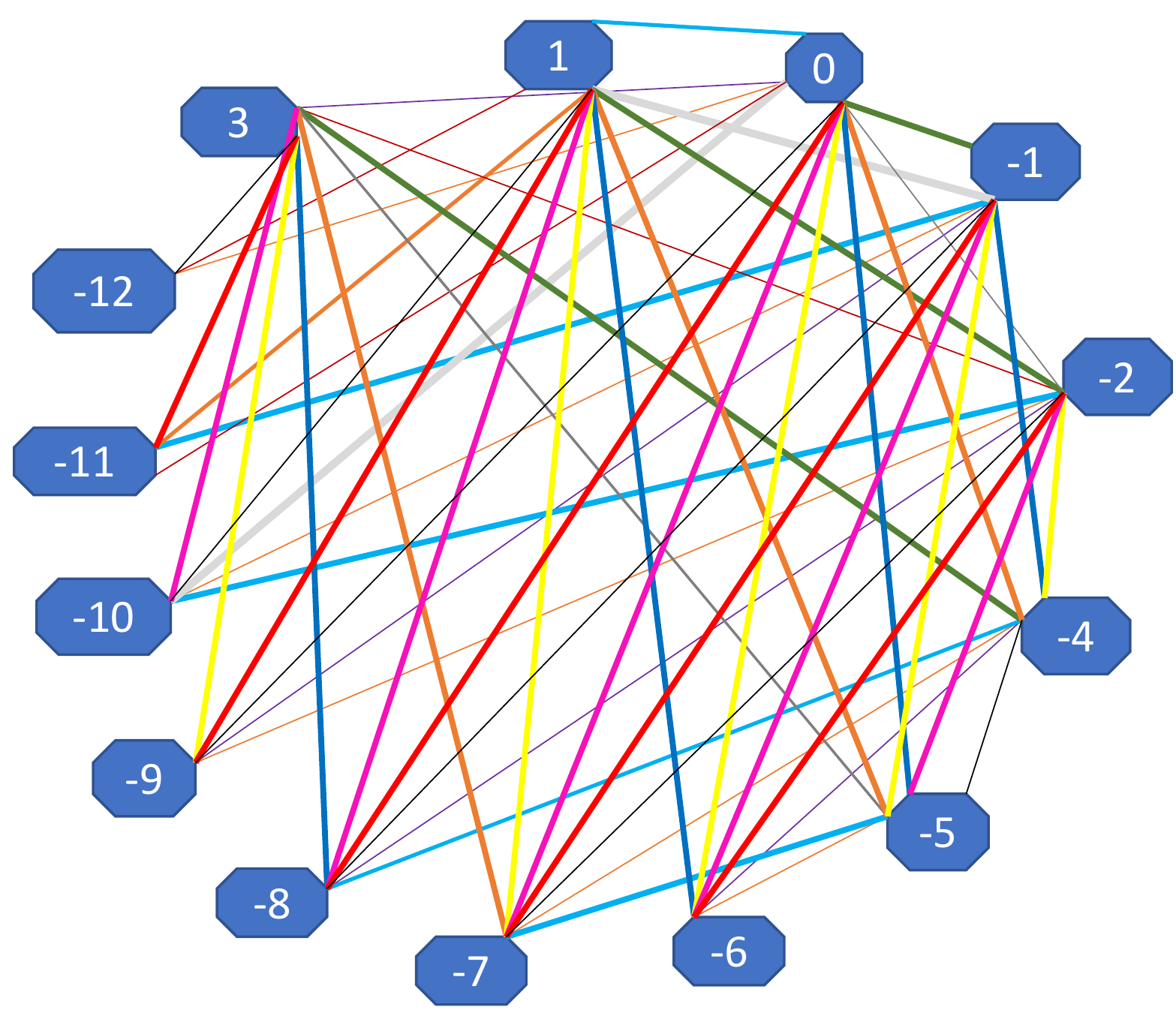}
	\end{minipage}}
\caption{Edge coloring}
\label{paper1figure4}
\end{figure}
Now, we choose $s=11$. By using Theorem \ref{0paper2theorem10}, we have the following color classes of $H^{-2,11}_{1,3}$:
\begin{eqnarray*}
&& \lbrace(0,-2),~(1,9),~(2,8),~(4,6)\rbrace \rightarrow \mbox{Blue},  \lbrace(0,11),~(1,10),~(2,9),(4,7),(5,6) \rbrace \rightarrow \mbox{Ivory},\\
&& \lbrace(0,10),~(-2,2)\rbrace \rightarrow \mbox{Brown},  \lbrace(0,1) \rbrace \rightarrow \mbox{Orange}, \lbrace(-2,4),(0,2) \rbrace \rightarrow \mbox{Grey},~
\lbrace(-2,6),(0,4)\rbrace \rightarrow \mbox{Purple},\\
&& \lbrace(-2,7),(0,5),(1,4)\rbrace \rightarrow \mbox{Green},  \lbrace(-2,8),(0,6),(1,5),(2,4)\rbrace \rightarrow \mbox{Yellow},\\
&& \lbrace(-2,9),(0,7),(1,6),(2,5)\rbrace \rightarrow \mbox{Pink}, 
\lbrace(-2,10),(0,8),(1,7),(2,6)\rbrace \rightarrow \mbox{Red},\\
&&
\lbrace(-2,11),(0,9),(1,8),(2,7),(4,5)\rbrace \rightarrow \mbox{Black}.
\end{eqnarray*}

So, edge chromatic index of the graph $H^{-2,11}_{1,3}$ is $11$.  Again, for $s=12$, $m=2$ and $j=3$ from Theorem \ref{paper2theorem13}, we compute the edge color classes as follows:
\begin{eqnarray*}
&& \lbrace(0,-3),~(1,9),~(2,8),~(4,6)\rbrace \rightarrow \mbox{Purple},  \lbrace(-1,11),~(0,12),~(1,10),(2,9),(4,7),(5,6) \rbrace \rightarrow \mbox{Musturd},\\
&& \lbrace(0,-1),(1,11),(2,10),(4,8),~(5,7)\rbrace \rightarrow \mbox{Cyan},
  \lbrace(0,10),(-1,1)  \rbrace \rightarrow \mbox{ivory}, \\
  &&\lbrace(-3,2),(-1,12),(0,11)  \rbrace \rightarrow \mbox{Maroon}, \lbrace(-3,4),(-1,2),(0,1)\rbrace \rightarrow \mbox{Olive green},\\
&& \lbrace(-3,5), (0,2)\rbrace \rightarrow \mbox{Grey},  \lbrace(-3,7), (-1,5),(0,4)\rbrace \rightarrow \mbox{Orange},\\
&& \lbrace(-3,8), (-1,6),(0,5),(1,4)\rbrace \rightarrow \mbox{Blue}, 
\lbrace(-3,9), (-1,7),(0,6),(1,5),(2,4)\rbrace \rightarrow \mbox{Yellow},\\
&&
\lbrace(-3,10), (-1,8),(0,7),(1,6),(2,5)\rbrace \rightarrow \mbox{Pink}, 
\lbrace(-3,11), (-1,9),(0,8),(1,7),(2,6)\rbrace \rightarrow \mbox{Red},\\
&&
\lbrace(-3,12), (-1,10),(0,9),(1,8),(2,7),(4,5)\rbrace \rightarrow \mbox{Black}.
\end{eqnarray*}
 Therefore, the edge chromatic index is $13$.  Hence, these are non perfect edge sum color graphs.  In figure \ref{paper1figure5}, we have plotted the edge coloring of these two graphs.
\begin{figure}[H]
  \subfloat[$H^{-2,11}_{1,3}$]{
	\begin{minipage}[c][1\width]{
	   0.3\textwidth}
	   \centering
	   \includegraphics[width=1\textwidth]{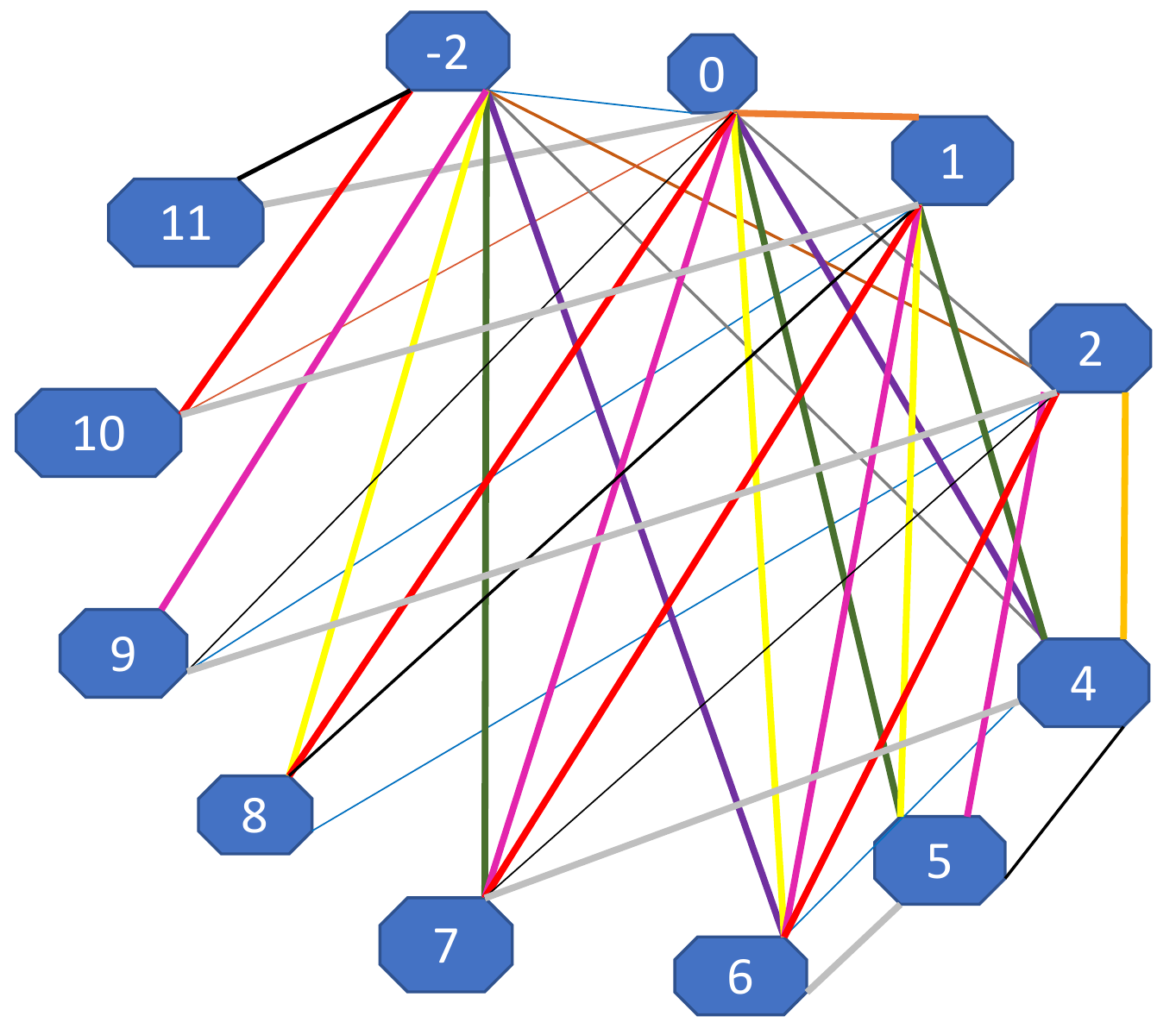}
	\end{minipage}}
 \hspace{4cm}	
  \subfloat[$H^{-3,12}_{2,3}$]{
	\begin{minipage}[c][1\width]{
	   0.3\textwidth}
	   \centering
	   \includegraphics[width=1\textwidth]{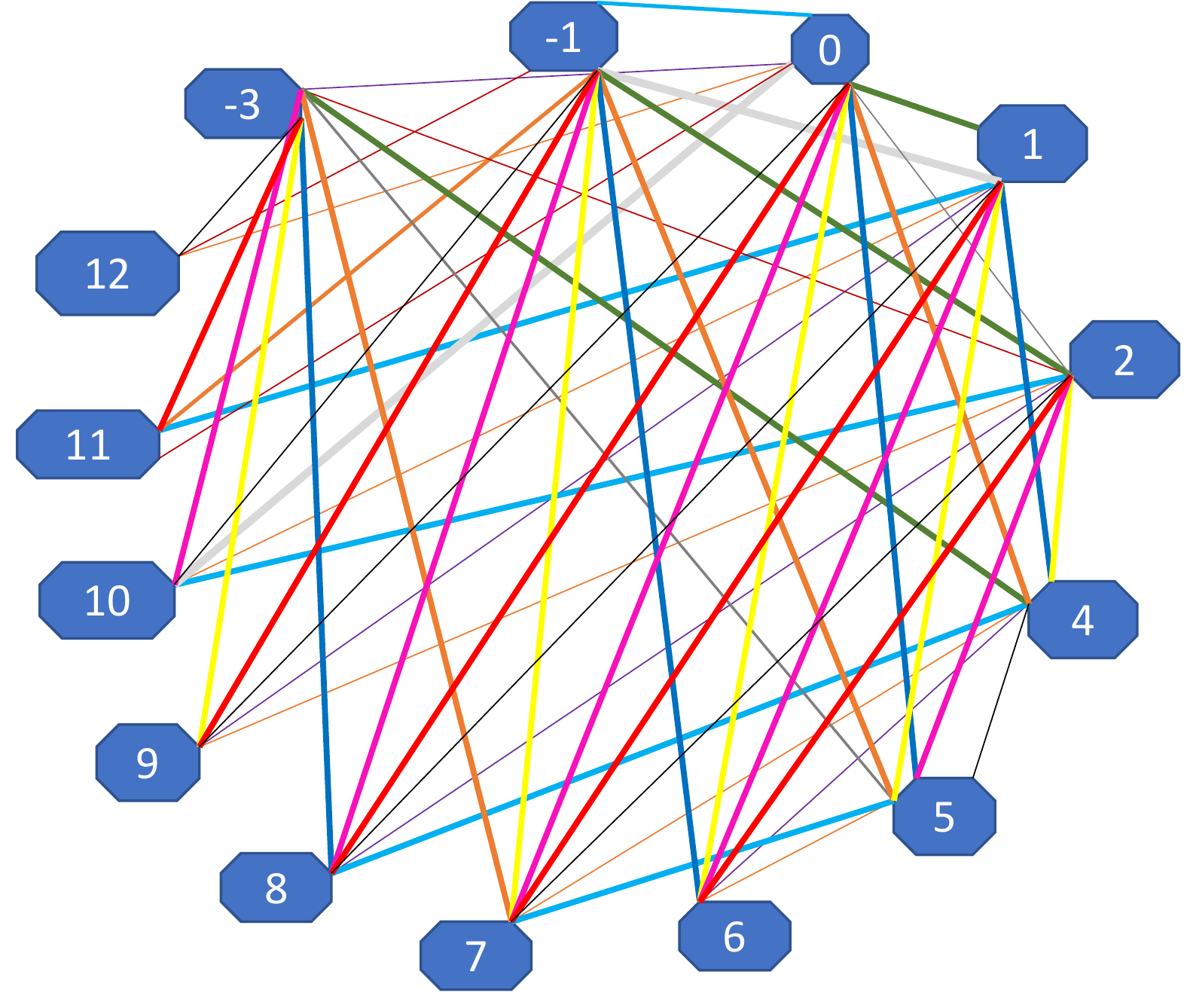}
	\end{minipage}}
\caption{Edge coloring}
\label{paper1figure5}
\end{figure}

\section{Conclusions}\label{P5Conclusions}
We considered the  families of integral sum graph   $H^{-i,s}_{m,j}$  for all $i,s, m,k\in \mathbb{N}$ such that $-i<0<s$. We successfully applied edge coloring and edge-sum coloring on these integral sum graphs. We derived the general formula of independent edge color classes for different values of $i$ and $s$. From the figures, we noticed that  no adjacent edges get the same color. We also compared these two techniques. We obvserved that edge coloring is more effective than edge sum coloring.    Numerical results guaranteed the validity of derived theoretical results.

\begin{thebibliography}{10}

\bibitem{ZB1979}
Z.~Baranyai.
\newblock The edge-coloring of complete hypergraphs {I}.
\newblock {\em Journal of Combinational Theory, Series B}, 26:276--294, 1979.

\bibitem{LWBSMH2021}
L.~W. Beineke, S.~M. Hegde, and V.~V. Kamalappan.
\newblock A survey of two types of labelings of graphs.
\newblock {\em Discrete Math. Lett.}, 6:8--18, 2021.

\bibitem{chen2006}
Z.~Chen.
\newblock On integral sum graphs.
\newblock {\em Discrete Math.}, 306:19--25, 2006.

\bibitem{CH1982}
R.~Cole and J.~Hopcroft.
\newblock On edge-coloring of bipartite graph.
\newblock {\em Siam J. Appl. Math.}, 11:540--546, 1982.

\bibitem{JF1970}
J.~Folkman.
\newblock Graphs with monochromatic complete subgraphs in every edge coloring.
\newblock {\em Siam J. Appl. Math.}, 18:19--24, 1970.

\bibitem{MKG1984}
M.~K. Goldberg.
\newblock Edge-coloring of multigraphs: Recoloring mark k. goldberg technique.
\newblock {\em Journal of Graph Theory}, 8:123--137, 1984.

\bibitem{SLH1986}
S.~L. Hakimi.
\newblock A generalization of edge-coloring in graphs.
\newblock {\em Journal of Graph Theory}, 10:139--154, 1986.

\bibitem{Hararay1969}
F.~Harary.
\newblock Graph theory.
\newblock {\em Addison Wesley}, 1969.

\bibitem{Hararay1990}
F.~Harary.
\newblock Sum graphs and difference graphs.
\newblock {\em Congr. Numer.}, 72:101--108, 1990.

\bibitem{Hararay1994}
F.~Harary.
\newblock Sum graphs over all integers.
\newblock {\em Discrete Math.}, 124:99--105, 1994.

\bibitem{IH1981}
I.~Holyer.
\newblock The {NP}-completness of edge coloring.
\newblock {\em Siam J. Appl. Math.}, 10:718--720, 1981.

\bibitem{AKRR1999}
A.~Kapoor and R.~Rizzi.
\newblock Edge-coloring bipartite graphs.
\newblock {\em Dipartimento di Matematica Pura ed Applicata}, 7:1--6, 1999.

\bibitem{MP2002}
L.~S. Melnikov and A.~V. Pyatkin.
\newblock Regular integral sum graphs.
\newblock {\em Discrete Math.}, 252:237--245, 2002.

\bibitem{TNSSVV2001}
T.~Nicholas, S.~Somasundaram, and V.~Vilfred.
\newblock Some results on sum graphs.
\newblock {\em J. Comb. Inf. Syst. Sci.}, 26:135--142, 2001.

\bibitem{APAS1997}
A.~Pancinesi and A.~Srinivasan.
\newblock Randomized distributed edge coloring via an extension of the
  {C}hernoff-{H}oeffding bounds.
\newblock {\em Siam J. Appl. Math.}, 267:350--368, 1997.

\bibitem{VLMF2012}
K.~Vilfred and L.~M. Florida.
\newblock Integral sum graphs and maximal integral sum graphs.
\newblock {\em Graph Theory Notes N. Y.}, 63:28--36, 2012.

\bibitem{VVLMF2013}
V.~Vilfred and L.~M. Florida.
\newblock New families of integral sum graph - edge sum class and chromatic
  integral sum number.
\newblock {\em Proceedings of the International Conference on Applied
  Mathematics and Theoretical Computer Science}, 2013.

\bibitem{WU2003}
J.~Wu, J.~Mao, and D.~Li.
\newblock New types of integral sum graphs.
\newblock {\em Discrete Math.}, 260:163--176, 2003.

\bibitem{Xu1994}
B.~Xu.
\newblock Note on integral sum graphs.
\newblock {\em Discrete Math.}, 194:285--294, 1999.

\bibitem{ZZLLJW2002}
Z.~Zhanga, L.~Liu, and J.~Wang.
\newblock Adjacent strong edge coloring of graphs.
\newblock {\em Applied Mathematics Letters}, 15:623--626, 2002.

\end{thebibliography}

\end{document}